\documentclass[dvips,12pt,a4paper]{amsart}

\usepackage[matrix,arrow,curve]{xy}

\usepackage{amsmath,amssymb}

\usepackage[dvips]{graphicx}

\def\M{{\mathcal{M}}}

\def\Z{{\mathbb Z}}
\def\C{{\mathbb C}}
\def\bL{{\mathbb L}}
\def\mcO{{\mathcal O}}

\def\P{{\mathbb P}}
\def\tpi{{\widetilde\pi}}
\def\tB{{\widetilde B}}

\newtheorem{proposition}{Proposition}[section]
\newtheorem{theorem}[proposition]{Theorem}
\newtheorem{corollary}[proposition]{Corollary}
\newtheorem{lemma}[proposition]{Lemma}

\title[$\M_{r,n}$ and Macmahon's formula]{The moduli space of sheaves and the generalization of MacMahon's formula}

\thanks{The author is partially supported by the grants RFBR-10-01-00678, NSh-8462.2010.1, the Vidi grant of NWO and by the Moebius Contest Foundation for Young Scientists}

\author{A. Buryak}

\address{Faculty of Mechanics and Mathematics, Moscow State University, 119991 Moscow, Russia and\newline\indent
Department of Mathematics, University of Amsterdam, P.~O.~Box 94248, 1090 GE Amsterdam, The Netherlands}
\email{buryaksh@mail.ru, a.y.buryak@uva.nl}
\begin{document}

\begin{abstract}
Recently M. Vuletic found a two-parameter generalization of the MacMahon's formula. In this note we show that certain ingredients of her formula have a clear interpretation in terms of the geometry of the moduli space of sheaves on the projective plane. 
\end{abstract}

\maketitle

\section{Introduction}

A plane partition is a Young diagram filled with positive integers that form nonincreasing rows and columns. For a plane partition $\pi$ one defines the weight $|\pi|$ to be the sum of all entries. Denote by $\mathcal P$ the set of all plane partitions. 

A generating function for the number of plane partitions is given by the famous MacMahon's formula (see e.g. \cite{Stanley}):
\begin{gather*}
\sum_{\pi\in\mathcal P} s^{|\pi|}=\prod_{n=1}^{\infty}\frac{1}{(1-s^n)^n}.
\end{gather*}

There are several generalizations of MacMahon's formula, see e.g. \cite{Ciucu,Vuletic}. In this paper we investigate the generalization of M. Vuletic from \cite{Vuletic}. For each plane partition $\pi$ she defined a rational function $F_{\pi}(q,t)$ and proved that
\begin{gather}\label{Vuletic formula}
\sum_{\pi\in\mathcal P}F_{\pi}(q,t)s^{|\pi|}=\prod_{n=1}^{\infty}\prod_{k=0}^{\infty}\left(\frac{1-ts^nq^k}{1-s^nq^k}\right)^n.
\end{gather}
Her proof was inspired by \cite{Okounkov} and \cite{Vuletic2}.

Let $\M_{r,n}$ be the framed moduli space of torsion free sheaves on $\P^2$ with rank $r$ and $c_2=n$. This is a smooth irreducible quasi-projective variety of dimension $2rn$. In the case $r=1$ it is isomorphic to the Hilbert scheme of $n$ points on the plane. The moduli space $\M_{r,n}$ has a simple quiver description and we recall it in Section \ref{subsection:quiver description}. There is a natural action of the two-dimensional torus $T=(\C^*)^2$ on $\M_{r,n}$. We refer the reader to the book \cite{Nakajima} for a more detailed discussion of the moduli space $\M_{r,n}$.   

We denote by $K_0(\nu_{\C})$ the Grothendieck ring of complex quasi-projective varieties. 

In this note we show that the coefficients $F_{\pi}(q,0)$ give the formulas for the the classes in $K_0(\nu_{\C})$ of the irreducible components of the fixed point set $\M_{r,n}^T$. 

We also show how to use the $T$-action on $\M_{r,n}$ to get a combinatorial identity, which is close to \eqref{Vuletic formula}.

We refer the reader to \cite{Buryak} for results about the Hilbert scheme of $n$ points on the plane close to this work.

\subsection{Definition of $F_{\pi}(q,t)$} 

For nonnegative integers $n$ and $m$ let 
\begin{gather*}
f(n,m)=\begin{cases}
					\prod\limits_{i=0}^{n-1}\frac{1-q^it^{m+1}}{1-q^{i+1}t^m},&n\ge 1,\\
					1,&n=0.
			 \end{cases}
\end{gather*}
Let $\pi\in\mathcal P$ be a plane partition and let $(i,j)$ be a box in its support (where the entries are nonzero). Let $\lambda,\mu$ and $\nu$ be the ordinary partitions defined by
\begin{align*}
&\lambda=(\pi_{i,j},\pi_{i+1,j+1},\ldots),\\
&\mu=(\pi_{i+1,j},\pi_{i+2,j+1},\ldots),\\
&\nu=(\pi_{i,j+1},\pi_{i+1,j+2},\ldots).
\end{align*}
For a box $(i,j)$ of $\pi$ let
\begin{gather*}
F_{\pi}(i,j)(q,t)=
\prod_{m=0}^{\infty}\frac{f(\lambda_1-\mu_{m+1},m)f(\lambda_1-\nu_{m+1},m)}{f(\lambda_1-\lambda_{m+1},m)f(\lambda_1-\lambda_{m+2},m)}.
\end{gather*}
An example is on Figure \ref{pic2}.

\begin{figure}[h]
\begin{center}
\includegraphics{ex2.1}
\end{center}
\caption{}
\label{pic2}
\end{figure}

For a plane partition $\pi$ the rational function $F_{\pi}(q,t)$ is defined by
\begin{gather*}
F_{\pi}(q,t)=\prod_{(i,j)\in\pi}F_{\pi}(i,j)(q,t).
\end{gather*} 

\subsection{Grothendieck ring of quasi-projective varieties}
Here we recall a definition of the Grothendieck ring $K_0(\nu_{\C})$ of complex quasi-projective varieties. It is an abelian group generated by the classes $[X]$ of all complex quasi-projective varieties $X$ modulo the relations:
\begin{enumerate}
\item if varieties $X$ and $Y$ are isomorphic, then $[X]=[Y]$;
\item if $Y$ is a Zariski closed subvariety of $X$, then $[X]=[Y]+[X\backslash Y]$.
\end{enumerate}  
The multiplication in $K_0(\nu_{\C})$ is defined by the Cartesian product of varieties: $[X_1]\cdot[X_2]=[X_1\times X_2]$. The class $\left[\mathbb A^1_{\C}\right]\in K_0(\nu_{\C})$ of the complex affine line is denoted by $\bL$.

\subsection{Moduli space of sheaves on $\P^2$} 

In this section we show a geometric meaning of the functions $F_{\pi}(q,0)$.
 
The moduli space $\M_{r,n}$ is defined by
\begin{gather*}
\M_{r,n}=\left.\left\{(E,\Phi)\left|
\begin{smallmatrix}
\text{$E$: a torsion free sheaf on $\P^2$}\\
rank(E)=r, c_2(E)=n\\
\text{$\Phi\colon E|_{l_{\infty}}\xrightarrow{\sim}\mcO^{\oplus r}_{l_{\infty}}$: a framing at infinity}
\end{smallmatrix}\right.\right\}\right/{\text{isomorphism}},
\end{gather*}
where $l_{\infty}=\{[0:z_1:z_2]\in\P^2\}\subset\P^2$ is the line at infinity. 

The torus $T=(\C^*)^2$ acts on $\C^2$ by scaling the coordinates, $(t_1,t_2)(x,y)=(t_1x,t_2y)$. This action lifts to the $T$-action on the moduli space $\M_{r,n}$. We will prove that the irreducible components of the variety $\M_{r,n}^T$ are enumerated by plain partitions $\pi$ such that $|\pi|=n$ and $\pi_{0,0}\le r$. We denote by $\M_{r,n}^T(\pi)$ the corresponding irreducible components. We use the notation $[N]!_{q}=\prod_{i=1}^N(1-q^i)$. We will prove the following statement.

\begin{theorem}\label{theorem:irreducible component}
Let $\pi$ be a plane partition such that $|\pi|=n$ and $\pi_{0,0}\le r$, then
\begin{gather*}
\left[\M_{r,n}^T(\pi)\right]=\frac{[r]!_{\bL}}{[r-\pi_{0,0}]!_{\bL}}\prod_{i,j\ge 0}\frac{[\pi_{i,j}-\pi_{i+1,j+1}]!_{\bL}}{[\pi_{i,j}-\pi_{i+1,j}]!_{\bL}[\pi_{i,j}-\pi_{i,j+1}]!_{\bL}}.
\end{gather*}
\end{theorem}
Consider the map $\psi\colon\M_{r,n}\to\M_{r+1,n}$ defined by $E\mapsto E\oplus\mcO_{\P^2}$, where $E$ is a sheaf. The map $\psi$ is an embedding of $\M_{r,n}$ into $\M_{r+1,n}$. This embedding induces an embedding of $\M_{r,n}^T(\pi)$ into $\M_{r+1,n}^T(\pi)$. We denote by $\M_{\infty,n}^{T}(\pi)$ the limit space. The space $\M_{\infty,n}^T(\pi)$ has infinite dimension, but using a generalization of the ring $K_0(\nu_{\C})$ the class $\left[\M_{\infty,n}^T(\pi)\right]$ can be defined. The class $\left[\M_{\infty,n}^T(\pi)\right]$ is an infinite series in $\bL$ equal to $\lim_{r\to\infty}\left[\M_{r,n}^T(\pi)\right]$. From Theorem~\ref{theorem:irreducible component} it follows that 
\begin{gather*}
\left[\M_{\infty,n}^T(\pi)\right]=\prod_{i,j\ge 0}\frac{[\pi_{i,j}-\pi_{i+1,j+1}]!_{\bL}}{[\pi_{i,j}-\pi_{i+1,j}]!_{\bL}[\pi_{i,j}-\pi_{i,j+1}]!_{\bL}}.
\end{gather*}

The following statement shows a geometric interpretation of the series $F_{\pi}(q,0)$.

\begin{theorem}
$F_{\pi}(\bL,0)=\left[\M_{\infty,n}^T(\pi)\right]$.
\end{theorem}
\begin{proof}
Direct computation.
\end{proof}
    
Using the result of M. Vuletic we obtain the following corollary.

\begin{corollary}
\begin{gather*}
\sum_{n\ge 0}\left[\M_{\infty,n}^T\right]t^n=\prod_{\substack{i\ge 0\\j\ge 1}}\frac{1}{(1-\bL^it^j)^j}.
\end{gather*}  
\end{corollary}

\subsection{Combinatorial identity}

Here we give an application of Theorem \ref{theorem:irreducible component}. In Section \ref{section:combinatorial identity} we use the $T$-action on $\M_{r,n}$ to get a decomposition of $\M_{r,n}$ into locally closed subvarieties. These subvarieties are locally trivial bundles over the varieties $\M_{r,n}^T(\pi)$. Then Theorem \ref{theorem:irreducible component} can be applied to obtain the following statement.
 
\begin{theorem}\label{theorem:combinatorial identity}
\begin{gather*}
\sum_{\substack{\pi\in\mathcal P\\\pi_{0,0}\le r}}t^{|\pi|}\frac{[r]!_{q}}{[r-\pi_{0,0}]!_{q}}q^{\chi(\pi)}F_{\pi}(q,0)=\prod_{\substack{n\ge 1\\1\le m\le r}}\frac{1}{1-q^mt^n},
\end{gather*}
where $\chi(\pi)=\sum_{i,j\ge 0}\pi_{i,j}(\pi_{i,j}-\pi_{i,j+1})$. In particular 
\begin{gather*}
\sum_{\pi\in\mathcal P}t^{|\pi|}q^{\chi(\pi)}F_{\pi}(q,0)=\prod_{m,n\ge 1}\frac{1}{1-q^mt^n}.
\end{gather*}
\end{theorem} 

\subsection{Organization of the paper}

In Section \ref{section:moduli space} we recall the quiver description of the moduli space $\M_{r,n}$. Then we use it to describe the irreducible components of the variety $\M_{r,n}^T$. Finally, we prove Theorem \ref{theorem:irreducible component}. Section \ref{section:combinatorial identity} contains the proof of Theorem \ref{theorem:combinatorial identity}.

\subsection{Acknowledgments}

The author is grateful to S. M. Gusein-Zade and B. L. Feigin for suggesting the area of research. The author is grateful to S. Shadrin for useful discussions. 

\section{Moduli space of sheaves on $\P^2$}\label{section:moduli space}

\subsection{Quiver description of $\M_{r,n}$}\label{subsection:quiver description}

The variety $\M_{r,n}$ has the following quiver description (see e.g. \cite{Nakajima}).
\begin{gather*}
\M_{r,n}\cong\left.\left\{(B_1,B_2,i,j)\left|
\begin{smallmatrix}
1) [B_1,B_2]+ij=0\\
2) \text{(stability) There is no subspace} \\
\text{$S\subsetneq\C^n$ such that $B_{\alpha}(S)\subset S$ ($\alpha=1,2$)}\\
\text{and $im(i)\subset S$} 
\end{smallmatrix}\right.\right\}\right/GL_n(\C),
\end{gather*}
where $B_1,B_2\in End(\C^n), i\in Hom(\C^r,\C^n)$ and $j\in Hom(\C^n,\C^r)$ with the action given by 
\begin{gather*}
g\cdot(B_1,B_2,i,j)=(gB_1g^{-1},gB_2g^{-1},gi,jg^{-1})
\end{gather*} 
for $g\in GL_n(\C)$. 

In the quiver description the map $\psi\colon\M_{r,n}\to\M_{r+1,n}$ is induced by the coordinate embedding of $\C^r$ into $\C^{r+1}$.

\subsection{Irreducible components of $\M^T_{r,n}$}\label{subsection:irreducible components}

In terms of Section \ref{subsection:quiver description} the $T$-action on $\M_{r,n}$ is given by (see e.g. \cite{Nakajima2})
$$
(t_1,t_2)\cdot[(B_1,B_2,i,j)]=[(t_1B_1,t_2B_2,i,t_1t_2j)].
$$ 
By definition, $[(B_1,B_2,i,j)]\in\M_{r,n}$ is a fixed point if and only if there exists a homomorphism $\lambda\colon T\to GL_n(\C)$ satisfying the following conditions: 
\begin{align}\label{formula:fixed point}
t_1B_1&=\lambda(t)^{-1}B_1\lambda(t),\notag\\
t_2B_2&=\lambda(t)^{-1}B_2\lambda(t),\\
i&=\lambda(t)^{-1}i,\notag\\
t_1t_2j&=j\lambda(t).\notag
\end{align} 
Suppose that $[(B_1,B_2,i,j)]$ is a fixed point. Then we have the weight decomposition of $\C^n$ with respect to $\lambda(t)$, i.e. $\C^n=\bigoplus_{k,l} V_{k,l}$, where $V_{k,l}=\{v\in \C^n|\lambda(t)\cdot v=t_1^kt_2^lv\}$. From the conditions \eqref{formula:fixed point} it follows that the only components of $B_1$, $B_2$, $i$ and $j$ which might survive are 
\begin{align*}
B_1&\colon V_{k,l}\to V_{k-1,l},\\
B_2&\colon V_{k,l}\to V_{k,l-1},\\
i&\colon \C^r\to V_{0,0},\\
j&\colon V_{1,1}\to \C^r.
\end{align*}
From the stability condition it follows that 
\begin{align*}
&V_{k,l}=0, \text{if $k>0$ or $l>0$},\\
&j=0,\\
&\dim V_{0,0}\le r,\\
&\dim V_{k,l}\ge \dim V_{k-1,l},\\
&\dim V_{k,l}\ge \dim V_{k,l-1}. 
\end{align*}
We see that the numbers $(\dim V_{-k,-l})_{k,l\ge 0}$ form a plane partition. 

Let $(\pi_{i,j})_{i,j\ge 0}$ be a plane partition such that $|\pi|=n$ and $\pi_{0,0}\le r$. Let $\M_{r,n}^T(\pi)$ be the subset of points from $\M_{r,n}^T$ such that $\dim V_{-k,-l}=\pi_{k,l}$. It is easy to see that $\M_{r,n}^T(\pi)$ is a closed subvariety of $\M_{r,n}^T$ and 
$$
\M_{r,n}^T=\coprod_{\substack{\pi\in\mathcal P\\ \pi_{0,0}\le r}}\M_{r,n}^T(\pi).
$$ 

We see that the variety $\M_{r,n}^T(\pi)$ has the following quiver description. Let $V_{-k,-l}=\C^{\pi_{k,l}}$. Then
\begin{multline*}
\M_{r,n}^T(\pi)\cong\\
\left.\left\{((B_{1,k,l})_{k,l\ge 0},(B_{2,k,l})_{k,l\ge 0},i)\left|
\begin{smallmatrix}
1) B_{1,k,l+1}B_{2,k,l}=B_{2,k+1,l}B_{1,k,l}\\
2) \text{$B_{1,k,l}, B_{2,k,l},i$ are surjective}
\end{smallmatrix}\right.\right\}\right/\prod_{k,l}GL_{\pi_{k,l}},
\end{multline*}
where $B_{1,k,l}\in Hom(V_{-k,-l},V_{-k-1,-l})$, $B_{2,k,l}\in Hom(V_{-k,-l},V_{-k,-l-1})$ and $i\in Hom(\C^r,V_{0,0})$ (see Figure \ref{quiver1}). 

\begin{figure}
$$
\xymatrix{
                      &                                    &                                          &                             & \\
                      & V_{0,-2}\ar[r]^-{B_1} \ar[u]^-{B_2} &                                          &                             & \\
                      & V_{0,-1}\ar[r]^-{B_1} \ar[u]^-{B_2} & V_{-1,-1}\ar[r]^-{B_1} \ar[u]^-{B_2}       &                             & \\
\C^r\ar[r]^-(0.45){i} & V_{0,0}\ar[r]^-{B_1} \ar[u]^-{B_2} & V_{-1,0}\ar[r]^-{B_1} \ar[u]^-{B_2} & V_{-2,0}\ar[r]^-{B_1} \ar[u]^-{B_2}&
}
$$
\caption{The quiver description of $\M_{r,n}^T(\pi)$}
\label{quiver1}
\end{figure}
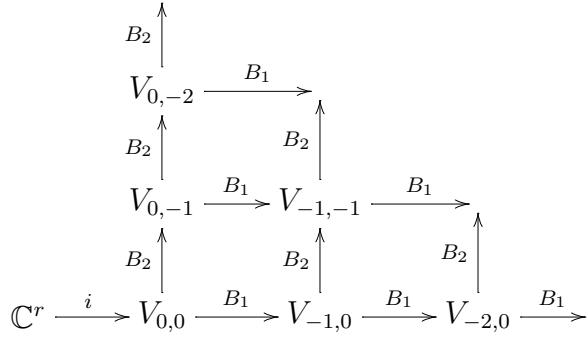

From Theorem~\ref{theorem:irreducible component} it follows that the class of $\M_{r,n}^T(\pi)$ is a polynomial in $\bL$. The coefficient of $\bL^i$ in this polynomial is equal to the $2i$-th Betti number of $\M_{r,n}^T(\pi)$. By Theorem~\ref{theorem:irreducible component}, the coefficient of $\bL^0$ is equal to~$1$, hence the variety $\M_{r,n}^T(\pi)$ is irreducible.

\subsection{Proof of Theorem \ref{theorem:irreducible component}}\label{subsection:class}

Consider partitions $\mu=(\mu_1,\mu_2,\ldots,\mu_k)$ and $\nu=(\nu_1,\nu_2,\ldots,\nu_k)$ such that $\mu_i\ge\nu_i$. Let $h=(h_i)_{1\le i\le k-1}$, where $h_i\in Hom(\C^{\nu_i},\C^{\nu_{i+1}})$ are surjective homomorphisms. We consider the variety $N_{\mu,\nu,h}$ defined by
$$
N_{\mu,\nu,h}=\left\{((f_i)_{1\le i\le k-1},(g_i)_{1\le i\le k})\left|
\begin{smallmatrix}
1) g_{i+1}f_i=h_i g_i\\
2) \text{$f_i$ and $g_i$ are surjective}
\end{smallmatrix}
\right.\right\},
$$
where $f_i\in Hom(\C^{\mu_i},\C^{\mu_{i+1}})$ and $g_i\in Hom(\C^{\mu_i},\C^{\nu_i})$. It is easy to see that the variety $N_{\mu,\nu,h}$ is smooth.
\begin{lemma}\label{lemma:lemma1}
\begin{gather}\label{eq:lemma}
\left[N_{\mu,\nu,h}\right]=\frac{[\mu_1]!_{\bL}\left[GL_{\nu_1}\right]}{[\nu_1]!_{\bL}[\mu_1-\nu_1]!_{\bL}}\prod_{i=1}^{k-1}\frac{[\mu_i-\nu_{i+1}]!_{\bL}\left[GL_{\mu_{i+1}}\right]}{[\mu_i-\mu_{i+1}]!_{\bL}[\mu_{i+1}-\nu_{i+1}]!_{\bL}}.
\end{gather}
\end{lemma}
\begin{proof}
The proof is by induction on $k$. Suppose $k=1$, then $N_{\mu,\nu,h}$ is the variety of $\mu_1\times \nu_1$-matrices of the maximal rank. The class of this variety is equal to $\frac{[\mu_1]!_{\bL}\left[GL_{\nu_1}\right]}{[\mu_1-\nu_1]!_{\bL}[\nu_1]!_{\bL}}$.

Suppose $k\ge 2$. Obviously, we have $Ker(f_{k-1})\subset Ker(h_{k-1} g_{k-1})$. Let $Gr_M(V)$ be the Grassmanian of $M$-dimensional vector subspaces in a vector space $V$. Let $h'=(h_1,h_2,\ldots,h_{k-2})$, $\mu'=(\mu_1,\mu_2,\ldots,\mu_{k-1})$ and $\nu'=(\nu_1,\nu_2,\ldots,\nu_{k-1})$. Let 
$$
\widetilde N_{\mu',\nu,h}=\left\{(f',g',v)\left|(f',g')\in N_{\mu',\nu',h'},v\in Gr_{\mu_{k-1}-\mu_k}(Ker(h_{k-1} g_{k-1}'))\right.\right\}.
$$
We define the map $\phi\colon N_{\mu,\nu,h}\to \widetilde N_{\mu',\nu,h}$ by the following formula
$$
((f_i)_{1\le i\le k-1},(g_i)_{1\le i\le k})\mapsto ((f_i)_{1\le i\le k-2},(g_i)_{1\le i\le k-1},Ker(f_{k-1})). 
$$
It is easy to see that the map $\phi$ is a locally trivial bundle with the variety $GL_{\mu_k}$ as the fiber. Therefore, we have 
\begin{align*}
\left[N_{\mu,\nu,h}\right]&=\left[\widetilde N_{\mu',\nu,h}\right]\left[GL_{\mu_k}\right]=\left[N_{\mu',\nu',h'}\right]\left[Gr_{\mu_{k-1}-\mu_k}(\C^{\mu_{k-1}-\nu_k})\right]\left[GL_{\mu_k}\right]=\\
&=\left[N_{\mu',\nu',h'}\right]\frac{[\mu_{k-1}-\nu_k]!_{\bL}\left[GL_{\mu_k}\right]}{[\mu_{k-1}-\mu_k]!_{\bL}[\mu_k-\nu_k]!_{\bL}}.
\end{align*}
By the induction hypothesis this is equal to the right-hand side of~\eqref{eq:lemma}.
\end{proof}
It is easy to see that the varieties $N_{\mu,\nu,h}$ are isomorphic for different choices of the maps $h_i$.
Let 
$$
N(\pi)=\left\{((B_{1,i,j})_{i,j\ge 0},(B_{2,i,j})_{i,j\ge 0})\left|
\begin{smallmatrix}
1) B_{1,i,j+1}B_{2,i,j}=B_{2,i+1,j}B_{1,i,j}\\
2) \text{$B_{\alpha,i,j}$ are surjective}
\end{smallmatrix}\right.\right\},
$$
where $B_{1,i,j}\in Hom(V_{-i,-j},V_{-i-1,-j})$, $B_{2,i,j}\in Hom(V_{-i,-j},V_{-i,-j-1})$. 

Theorem \ref{theorem:irreducible component} is equivalent to the following equation
\begin{gather}\label{eq:equation2}
\left[N(\pi)\right]=\frac{[\pi_{0,0}]!_{\bL}}{\left[GL_{\pi_{0,0}}\right]}\prod_{i,j\ge 0}\frac{[\pi_{i,j}-\pi_{i+1,j+1}]!_{\bL}\left[GL_{\pi_{i,j}}\right]}{[\pi_{i,j}-\pi_{i+1,j}]!_{\bL}[\pi_{i,j}-\pi_{i,j+1}]!_{\bL}}.
\end{gather}

Let $l$ be the number of rows in the support of $\pi$. The proof of \eqref{eq:equation2} is by induction on $l$. Suppose $l=0$, then \eqref{eq:equation2} is obvious. Suppose $l\ge 1$. Let $\widetilde\pi$ be the plane partition defined by $\widetilde\pi_{i,j}=\pi_{i,j+1}$. Consider a point $((B_{1,i,j}),(B_{2,i,j}))\in N(\pi)$. If we forget the maps $B_{1,i,0}$ and $B_{2,i,0}$, then we obtain a point from $N(\widetilde\pi)$. This defines the map $\rho\colon N(\pi)\to N(\widetilde\pi)$. It is easy to see that the map $\rho$ is a locally trivial bundle. The fiber over a point $((\tB_{1,i,j}),(\tB_{2,i,j}))\in N(\tpi)$ is the variety $N_{\mu,\nu,h}$, where $\mu_i=\pi_{i,0}$, $\nu_i=\pi_{i,1}$ and $h_i=\tB_{1,i,0}$. Using the induction hypothesis and Lemma \ref{lemma:lemma1} we get
\begin{align*}
\left[N(\pi)\right]=&\left(\frac{[\pi_{0,1}]!_{\bL}}{\left[GL_{\pi_{0,1}}\right]}\prod_{i\ge 0, j\ge 1}\frac{[\pi_{i,j}-\pi_{i+1,j+1}]!_{\bL}\left[GL_{\pi_{i,j}}\right]}{[\pi_{i,j}-\pi_{i+1,j}]!_{\bL}[\pi_{i,j}-\pi_{i,j+1}]!_{\bL}}\right)\times\\
&\times\frac{[\pi_{0,0}]!_{\bL}\left[GL_{\pi_{0,1}}\right]}{[\pi_{0,1}]!_{\bL}[\pi_{0,0}-\pi_{0,1}]!_{\bL}}\prod_{i\ge 0}\frac{[\pi_{i,0}-\pi_{i+1,1}]!_{\bL}\left[GL_{\pi_{i+1,0}}\right]}{[\pi_{i,0}-\pi_{i+1,0}]!_{\bL}[\pi_{i+1,0}-\pi_{i+1,1}]!_{\bL}}=\\
&=\frac{[\pi_{0,0}]!_{\bL}}{\left[GL_{\pi_{0,0}}\right]}\prod_{i,j\ge 0}\frac{[\pi_{i,j}-\pi_{i+1,j+1}]!_{\bL}\left[GL_{\pi_{i,j}}\right]}{[\pi_{i,j}-\pi_{i+1,j}]!_{\bL}[\pi_{i,j}-\pi_{i,j+1}]!_{\bL}}.
\end{align*}
This completes the proof of Theorem \ref{theorem:irreducible component}.

\section{Proof of Theorem \ref{theorem:combinatorial identity}}\label{section:combinatorial identity}

Let $T_{a,b}=\{(t^a,t^b)\in T|t\in\C^*\}$, be a one dimensional subtorus of $T$. Consider the $T_{1,\alpha}$-action on $\M_{r,n}$, where $\alpha$ is a positive integer. Suppose $\alpha$ is big enough. Then the set of fixed points of the $T_{1,\alpha}$-action coincides with the set of fixed points of the $T$-action on $\M_{r,n}$. We define the map $\rho\colon\M_{r,n}\to\M_{r,n}^T$ by the following formula $\rho(p)=\lim_{t\to 0}t\cdot p$, where $p\in\M_{r,n}$ and $t\in T_{1,\alpha}$. Therefore, we have 
\begin{gather*}
\M_{r,n}=\coprod_{\substack{\pi\in\mathcal P\\\pi_{0,0}\le r}}\rho^{-1}(\M_{r,n}^T(\pi)).
\end{gather*} 
From \cite{B1,B2} it follows that $\rho^{-1}(\M_{r,n}^T(\pi))$ is a locally closed subvariety such that the map $\rho^{-1}\left(\M_{r,n}^T(\pi)\right)\xrightarrow{\rho}\M_{r,n}^T(\pi)$ is a locally trivial bundle with an affine space as the fiber. We denote by $d^+_{1,\alpha}(\pi)$ the dimension of the fiber. Therefore, we have
\begin{gather}\label{identity}
\left[\M_{r,n}\right]=\sum_{\substack{\pi\in\mathcal P\\\pi_{0,0}\le r}}\left[\M_{r,n}^T(\pi)\right]\bL^{d^+_{1,\alpha}(\pi)}.
\end{gather}
Let us prove that 
\begin{gather}\label{dimplus}
d^+_{1,\alpha}(\pi)=rn+\sum_{i,j\ge 0}\pi_{i,j}(\pi_{i,j}-\pi_{i,j+1}).
\end{gather}

The $T$-action on $\M_{r,n}$ is the part of the action of the $(r+2)$-dimensional torus $T\times (\C^*)^r$. In terms of Section \ref{subsection:quiver description} this action is given by (see e.g. \cite{Nakajima2})
\begin{gather*}
(t_1,t_2,e_1,e_2,\ldots,e_r)\cdot[(B_1,B_2,i,j)]=[(t_1B_1,t_2B_2,ie^{-1},t_1t_2ej)],
\end{gather*}
where $e=diag(e_1,e_2,\ldots,e_r)$ is the diagonal $r\times r$-matrix.

The set of fixed points of the $T\times(\C^*)^r$-action is finite and is parametrized by the set of $r$-tuples $D=(D_1,D_2,\ldots,D_r)$ of Young diagrams $D_i$, such that $\sum_{i=1}^r|D_i|=n$ (see e.g. \cite{Nakajima2}). It is easy to see that the fixed point corresponding to an $r$-tuple $D$ belongs to $\M_{r,n}^T(\pi)$, where $\pi_{i,j}=|\{\alpha|(i,j)\in D_{\alpha}\}|$. 

For a Young diagram $Y$ let 
\begin{align*}
&r_l(Y)=|\{(i,j)\in D|j=l\}|,\\
&c_l(Y)=|\{(i,j)\in D|i=l\}|.
\end{align*}
For a point $s=(i,j)\in\Z_{\ge 0}^2$ let
\begin{align*}
&l_Y(s)=r_j(Y)-i-1,\\
&a_Y(s)=c_i(Y)-j-1,
\end{align*}
see Figure \ref{pic1}. Note that $l_Y(s)$ and $a_Y(s)$ are negative if $s\notin Y$.

\begin{figure}[h]
\begin{center}
\includegraphics{ex1.1}
\end{center}
\caption{}
\label{pic1}
\end{figure}

Let $p$ be the fixed point of the $T\times(\C^*)^r$-action corresponding to an $r$-tuple $D$. Let $R(T\times(\C^*)^r)=\Z[t_1,t_2,e_1,e_2,\ldots,e_r]$ be the representation ring of $T\times(\C^*)^r$. Then the weight decomposition of $T_p(\M_{r,n})$ is given by~(see e.g. \cite{Nakajima2})
\begin{gather}\label{weight decomposition}
T_p(\M_{r,n})=\sum_{i,j=1}^r e_j e_i^{-1}\left(\sum_{s\in D_i}t_1^{-l_{D_j}(s)}t_2^{a_{D_i}(s)+1}+\sum_{s\in D_j}t_1^{l_{D_i}(s)+1}t_2^{-a_{D_j}(s)}\right).
\end{gather}

Let $p$ be an arbitrary fixed point of the $T\times(\C^*)^r$-action on $\M_{r,n}^T(\pi)$. Let $D$ be the corresponding $r$-tuple of Young diagrams. From \eqref{weight decomposition} it follows that 
\begin{multline*}
d_{1,\alpha}^{+}(\pi)=\sum_{i,j=1}^r(|\{s\in D_i|\alpha(a_{D_i}(s)+1)-l_{D_j}(s)>0\}|+\\
+|\{s\in D_j|-\alpha a_{D_j}(s)+l_{D_i}(s)+1>0\}|).
\end{multline*}
Since $\alpha$ is big, the right-hand side is equal to
$$
rn+\sum_{i,j=1}^r|\{s\in D_j|a_{D_j}(s)=0,s\in D_i\}|=rn+\sum_{i,j\ge 0}\pi_{i,j}(\pi_{i,j}-\pi_{i,j+1}).
$$ 
Thus, we have proved \eqref{dimplus}.

It is well known that 
\begin{gather}\label{generating series}
\sum_{n\ge 0}[\M_{r,n}]t^n=\prod_{m=1}^r\prod_{n\ge 1}\frac{1}{1-\bL^{rn+m}t^n},
\end{gather}
see e.g. \cite{Nakajima3}. If we combine Theorem \ref{theorem:irreducible component} and the equations \eqref{identity}, \eqref{dimplus} and \eqref{generating series}, we obtain the proof of Theorem \ref{theorem:combinatorial identity}.

\end{document}